\documentclass[12pt]{article}
\usepackage{amssymb,amsmath,amsthm}
\usepackage[margin=2cm]{geometry}
\usepackage{enumerate}
\usepackage{hyperref}
\usepackage{cleveref}
\usepackage{authblk}
\newtheorem{thrm}{Theorem}

\newtheorem{corollary}[thrm]{Corollary}
\newtheorem{conj}{Conjecture}
\newtheorem{claim}{Claim}[thrm]

\crefname{lem}{Lemma}{Lemmas}
\crefname{thm}{Theorem}{Theorems}
\crefname{cor}{Corollary}{Corollaries}
\crefname{prop}{Proposition}{Propositions}
\crefname{conj}{Conjecture}{Conjectures}
\crefname{claim}{Claim}{Claims}
\crefname{openproblem}{Open Problem}{Open Problems}
\crefformat{equation}{(#2#1#3)}
\Crefformat{equation}{Equation #2(#1)#3}

\def\eref#1{$(\ref{#1})$}
\def\sref#1{\S$\ref{#1}$}

\def\tref#1{Theorem~$\ref{#1}$}
\def\cjref#1{Conjecture~$\ref{#1}$}
\def\cyref#1{Corollary~$\ref{#1}$}

\makeatletter
\g@addto@macro\bfseries{\boldmath}
\makeatother

\renewcommand\emptyset{\varnothing}

\DeclareMathOperator{\crad}{cr}

\title{Covering radius in the Hamming permutation space}

\author[1,2]{Kevin Hendrey\thanks{Supported by the Institute for Basic Science, No. IBS-R029-C1, and by a Monash University Postgraduate Publication Award.}}
 \author[2]{Ian M. Wanless\thanks{Supported by the Australian Research Council grant DP150100506.\\ This work is forthcoming in {\it The European Journal of Combinatorics}. It is licensed under the Creative Commons Attribution-NonCommercial-NoDerivatives 4.0 International License. To view a copy of this license, visit \href{https://creativecommons.org/licenses/by-nc-nd/4.0/}{https://creativecommons.org/licenses/by-nc-nd/4.0/}}}
\affil[1]{\small Discrete Mathematics Group, Institute for Basic Science (IBS), Daejeon, Republic of Korea}
\affil[2]{\small School of Mathematics, Monash University, Vic 3800, Australia\authorcr
	\small\tt Email: kevinhendrey@ibs.re.kr, ian.wanless@monash.edu}

\date{}

\renewcommand{\geq}{\geqslant}
\renewcommand{\leq}{\leqslant}
\renewcommand{\ge}{\geqslant}
\renewcommand{\le}{\leqslant}
\def\sym{\mathcal{S}}
\def\mm{\mathcal{M}}
\def\V{\mathcal{V}}
\def\W{\mathcal{W}}
\def\iwb{\rightarrow\infty}

\begin{document}

\maketitle

\begin{abstract}
  Let $\sym_n$ denote the set of permutations of
  $\{1,2,\dots,n\}$.  The function $f(n,s)$ is defined to be the
  minimum size of a subset $S\subseteq\sym_n$ with the property that for
  any $\rho\in\sym_n$ there exists some $\sigma\in S$ such that the
  Hamming distance between $\rho$ and $\sigma$ is at most $n-s$.
  The value of $f(n,2)$ is the subject of a conjecture by K\'ezdy and Snevily,
  which implies several famous conjectures about latin squares.

  We prove that the odd $n$ case of the K\'ezdy-Snevily Conjecture implies
  the whole conjecture. We also show that $f(n,2)>3n/4$ for all $n$, that
  $s!< f(n,s)< 3s!(n-s)\log n$ for $1\le s\le n-2$ and that 
\[f(n,s)>\left\lfloor \frac{2+\sqrt{2s-2}}{2}\right\rfloor \frac{n}{2}\]
if $s\geq 3$.
\end{abstract}
  
\section{Introduction}\label{s:intro}

Given a finite metric space $(X,d)$, the \emph{covering radius} $\crad(S)$ of a subset
$S$ of $X$ is the minimum real number $r$ such that balls of radius
$r$ centred at the points in $S$ cover $X$. A covering code for
$(X,d)$ is a subset with covering radius at most some specified
value. For practical applications, it is generally desirable to have a
covering code with few elements.  We examine the covering radius
problem for $(\sym_n,d_H)$, where $\sym_n$ is the set of permutations
of $[n]=\{1,2,\dots,n\}$ and $d_H$ is the Hamming distance (the number
of positions in which a pair of permutations disagree). See
\cite{Qui06} for background on covering and packing problems in this
space.

The function $f(n,s)$ is defined to be the minimum size of a subset of
$\sym_n$ that has covering radius at most $n-s$.  It is not hard to
show that $f(n,1)=\lfloor n/2\rfloor+1$ for all $n$; see for example
\cite{CK03,CW05}. However, the $s=2$ case is already difficult, and
has interesting connections to the study of latin squares.  A {\em
  latin square} of order $n$ is an {$n\times n$} array of $n$ symbols
in which each symbol occurs exactly once in each row and column. In a
latin square of order $n$, a {\em partial transversal of length
  $\ell$} is a set of $\ell$ entries containing no pair of entries
that lie in the same row or column or share the same symbol. A {\em
  transversal} is a partial transversal of length $n$ and a {\em near
  transversal} is a partial transversal of length $n-1$. See
\cite{transurv} for a survey on transversals, including a history of
the following two famous and long-standing conjectures in the area:

\begin{conj}\label{c:Ryser}
Each latin square of odd order has a transversal.
\end{conj}

\begin{conj}\label{c:Brualdi}
Every latin square has a near transversal.
\end{conj}

The first of these conjectures is usually attributed to Ryser, while
the second is variously attributed to Brualdi, Ryser and Stein. In
contrast to \cref{c:Ryser}, there are at least $n^{n^{3/2}(1/2-o(1))}$
Latin squares of each even order $n$ that have no transversal
\cite{CW17}. The rows of any such Latin square form a set of
permutations that has covering radius exactly $n-2$ (see
\cite{CW05}). It follows that $f(n,2)\le n$ for all even $n$.  K\'ezdy
and Snevily made the following conjecture, motivated by the fact that it implies both
\cref{c:Ryser,c:Brualdi} (again, see \cite{CW05} for details).

\begin{conj}\label{cj:KS}
If $n$ is even, then $f(n,2)=n$; if $n$ is odd, then $f(n,2)>n$. 
\end{conj}

In \cite{CW05} it was shown that
$\lfloor n/2\rfloor+2\le f(n,2)\le 4n/3+O(1)$ for all $n$.  In \cite{WZ13} it
was shown that $f(n,2)\le n+O(\log n)$ for all $n$ and that
$f(n,2)\le n+2$ whenever $n$ is divisible by $3$. In the next section we show
that it suffices to prove that $f(n,2)>n$ for odd $n$
(\cref{cj:KS,c:Ryser,c:Brualdi} would all follow). In \cyref{cy:3n/4},
we also significantly improve the lower bound on $f(n,2)$.
In the final section of the paper, we find new upper and lower
bounds on $f(n,s)$ for general $s$.

\section{Remarks on the K\'ezdy-Snevily conjecture}\label{s:KS}


In this section we consider the case at the heart of the
K\'ezdy-Snevily conjecture, namely $f(n,2)$. We will prove a new
lower bound. But first we show that half of \cjref{cj:KS} implies
the other half.

A set of permutations $S\subseteq\sym_n$ is {\em transitive} if for
every $x,y\in[n]$ there exists $\sigma\in S$ such that $\sigma(x)=y$.

\begin{thrm}\label{t:mono}
  If $S$ is a non-transitive subset of $\sym_n$ with
  $\crad(S)\le n-s$, 
  then $f(n-1,s)\leq |S|$.
\end{thrm}

\begin{proof}
  Since $S$ is non-transitive, there exists $x$ and $y$ in $[n]$
  such that $p(x)\neq y$ for all $p \in S$. By appropriate relabelling, 
  we may assume that $x=y=n$.
Consider the function $g$ which maps each permutation $p\in \sym_n$ to
the unique permutation $q\in \sym_{n-1}$ for which
\[ 
q(x)= \left\{ \begin{array} {ll}
  p(x) & \mathrm{if} \;\;p(x)\neq n, \\ 
  p(n) & \mathrm{if} \;\;p(x)=n,
\end{array} \right. 
\]
for all $x\in[n-1]$.
Let $p'$ be an arbitrary permutation in $\sym_{n-1}$, and let $p\in\sym_n$ be
the unique permutation such that $p(n)=n$ and $g(p)=p'$. Since
$\crad(S)\leq n-s$, there is some $q\in S$ such that $d_H(p,q)\leq n-s$.
Since $q\in S$, we know $q(n)\neq n$. Hence, there are at least
$s$ choices of $i\in[n-1]$ for which $q(i)=p(i)\neq n$.
It follows that $d_H(p',g(q))\leq n-1-s$. Therefore $\{g(q):q\in S\}$
is a subset of $\sym_{n-1}$ of size at most $|S|$ and covering radius
at most $n-1-s$.
\end{proof}

Part of the reason to be interested in \tref{t:mono} is because it
gives a weak version of monotonicity for $f(n,s)$ in its first
parameter.  It is known that $f(n,s)$ is not actually monotonic in
$n$, since $f(3,2)=6>f(4,2)=4$. On the other hand, $f(n,s)$ increases
monotonically in $s$ by definition. Another reason for interest in
\tref{t:mono} is the following implication:

\begin{corollary}
  If \cjref{cj:KS} is true for odd $n$, then it is true for all $n$.
\end{corollary}

\begin{proof}
Assume 
that there is some positive integer $k$ for which $f(2k,2)\ne2k$.
As was mentioned in \sref{s:intro}, we know that $f(2k,2)\le2k$, so we are
assuming that $f(2k,2)<2k$.
This implies that there is some $S\subseteq \sym_{2k}$ with $|S|<2k$ and
$\crad(S)\leq 2k-2$.  A transitive subset of $\sym_{2k}$ must contain at
least $2k$ permutations, since for any $x\in[2k]$ there
must be a permutation in the set which maps 1 to $x$. Therefore $S$ is
non-transitive, so by \tref{t:mono}, $f(2k-1,2)\leq |S|\leq2k-1$.
We have shown that if \cjref{cj:KS} fails for $n=2k$ then it also fails for $n=2k-1$.
The result follows.
\end{proof}

Our next aim is to improve the lower bound on $f(n,2)$ from
\cite{CW05}. To do this we will take a graph theoretic approach (which
we will also use in \sref{s:gens} to find a lower bound on $f(n,s)$
for $s\ge3$). For any graph $G$ we will denote the vertices and edges of
$G$ by $V(G)$ and $E(G)$, respectively.
For any set of permutations $S\subseteq\sym_n$ we colour the
edges of the complete bipartite graph $G=K_{n,n}$ on vertex sets
$\V=\{v_1,\dots, v_n\}$ and $\W=\{w_1,\dots,w_n\}$. For each $p\in S$
and $i\in[n]$ we give the edge from $v_i$ to $w_{p(i)}$ a colour that
is unique to $p$.  The edges with any particular colour form a perfect
matching in $G$, and we may think of $S$ as corresponding to a set
$\mm_S$ of perfect matchings. Each edge of $G$ will receive a number
of colours equal to the number of those perfect matchings that it is
in.  An edge is {\it blank} if it has no colour, {\it monochromatic}
if it has exactly one colour and {\it polychromatic} if it has at
least two colours.  A matching $M$ of $G$ is {\it $k$-light} (with
respect to $\mm_S$) if no edge in $M$ has more than $k$ colours, and
{\it $k$-rainbow} (with respect to $\mm_S$) if no colour occurs on
more than $k$ edges of $M$. We write {\it light} and {\it rainbow} for
1-light and 1-rainbow respectively.

To show that $f(n,s)>k$ we must show that for an arbitrary set
$S\subseteq\sym_n$ with $|S|=k$ there is an $(s-1)$-rainbow perfect matching of $G$.
In the $f(n,2)$ case, we will also insist that the matching is light, because it assists the
proof. We will locate the required matching by taking a matching $M$
which is close to what we want, then arguing that if it is not already
what we want, then we can improve it. The improvement will come via
the common technique of switching on an alternating path or cycle. A
path/cycle $X$ is {\it alternating} (with respect to $M$) if it has
the property that among any two consecutive edges of $X$, precisely one of
them is in $M$.  To {\it switch} $M$ with respect to $X$, we remove $E(X)\cap
M$ from $M$ and replace these edges with $E(X)\setminus M$.


\begin{thrm}\label{t:newfn2}
Let $G$ be the complete bipartite graph $K_{n,n}$ and let $\mm_S$ be a set
of at most $3n/4$ perfect matchings of $G$ (not necessarily
disjoint). There is a perfect matching of $G$ that
is light and rainbow with respect to $\mm_S$
and which contains a maximum sized $0$-light matching of $G$.
\end{thrm}

The remainder of this section is devoted to proving \tref{t:newfn2}
via a sequence of intermediate claims.  Let $M$ be a light rainbow
matching of $G$ with as many blank edges as possible and, subject to
this restriction, as many edges as possible. Aiming for a
contradiction, assume that $M$ is not a perfect matching. Without loss
of generality, $v_1$ and $w_1$ are unmatched by $M$. Let $A_1$ be the
set of vertices $v$ such that there is an even length alternating path
$P_v$ of blank edges from $v_1$ to $v$.  By construction, $P_v$ must
begin with an edge that is not in $M$ and end with a (blank) edge in
$M$. Also, every vertex in $\V$ that lies on $P_v$ must itself be in
$A_1$.  Let $A_2$ be the set of vertices $w$ such that there is an
even length alternating path of blank edges from $w_1$ to $w$. The
properties of $A_2$ are analogous to those of $A_1$.

Of the edges incident with a vertex $u\in V(G)$, let $\beta(u)$ be the
number of blank edges, $\pi(u)$ the number of polychromatic edges and
let $\mu(u)=n-\beta(u)-\pi(u)$ be the number of monochromatic edges.
For $i\in \{1,2\}$ and $u\in A_i$ let $\mu'(u)$ be the number of
monochromatic edges incident to $u$ that are part of an alternating
path of length 2 from $u$ to some vertex in $A_i$ (by definition,
the first edge of such a path cannot be in $M$, since no vertex
of $A_i$ is matched by a monochromatic edge of $M$).

\begin{claim}\label{sizeA} $|A_1|\geq 1+\beta(v)+\mu'(v)$ for all $v\in A_1$ 
and $|A_2|\geq 1+\beta(w)+\mu'(w)$ for all $w\in A_2$.
\end{claim}

\begin{proof}
Let $vw_i$ be a blank edge incident to some $v\in A_1$. By the
definition of $A_1$, there is an even length alternating path $P$ of
blank edges from $v_1$ to $v$. Since $G$ is bipartite, we can find
$P'$, an odd length alternating path of blank edges from $v_1$ to
$w_i$ (either $P'=Pw_i$ or $P'$ is some subpath of $P$). By our choice
of $M$, it cannot be increased by switching on $P'$, so $w_i$ is
matched by $M$ to some vertex $v_j$. Note that the first edge and last
edge of $P'$ are not in $M$, since $P'$ is an alternating path of odd
length and $v_1$ is not matched by $M$.
Hence $w_iv_j$ is not an edge of $P'$ and $v_j\notin V(P')$,
since every internal vertex of $P'$ is incident to an edge in
$M\cap E(P')$. It follows that $w_iv_j$ is blank, since otherwise the
matching obtained from $M$ by switching on $P'v_j$ would be a light
rainbow matching of the same size as $M$ but with more blank edges
than $M$. Hence, $P'v_j$ is an even length alternating path of blank
edges from $v_1$ to $v_j$, so $v_j\in A_1$.
Note that $P'v_j$ is simply a subpath of $P$ in the case when $P'$ is a subpath of $P$.

We have shown that each of the $\beta(v)$ vertices that are joined to $v$
by a blank edge is matched by $M$ to a vertex in $A_1$.
By definition, there are $\mu'(v)$ vertices that are joined to $v$ by a
monochromatic edge and matched  by $M$ to a vertex in $A_1$.
Also, $v_1$ is trivially in $A_1$, and is not matched by $M$.
Hence, $|A_1|\geq 1+\beta(v)+\mu'(v)$.
By symmetry $|A_2|\geq 1+\beta(w)+\mu'(w)$ for all $w\in A_2$.
\end{proof}

\begin{claim}\label{sizeb}
For $u\in V(G)$ we have $\beta(u)\geq n/4+\pi(u)$.
\end{claim}

\begin{proof}
There are at most $3n/4$ colours assigned to edges incident to $u$. No
colour is assigned to multiple edges incident to $u$, since each
colour induces a perfect matching. At least two colours are
assigned to each polychromatic edge, by definition. Hence, 
$\mu(u)\le 3n/4-2\pi(u)$, and
$$\beta(u)=n-\pi(u)-\mu(u)\geq n-\pi(u)-\left(\frac{3n}{4}-2\pi(u)\right)=\frac{n}{4}+\pi(u).\qedhere$$
\end{proof}

\begin{claim}\label{C}
  There are at most  $n+1-|A_1|-|A_2|$ edges in $M$ with no endpoint in $A_1\cup A_2$.
\end{claim}

\begin{proof}
There are at most $n-1$ edges in $M$. Every vertex in
$A_1\cup A_2\setminus \{v_1,w_1\}$ is an endpoint of an edge in $M$. No edge of
$M$ is between a vertex $v\in A_1$ and a vertex $w\in A_2$, since
otherwise we could find a blank alternating path from $v_1$ to $w_1$,
by the definition of $A_1$ and $A_2$. Switching on this path would increase the
number of blank edges in $M$, which is a contradiction.
\end{proof}

Let $C$ be the set of colours assigned to edges in $M$ and let
$\overline{C}$ be the set of colours not assigned to any edge in
$M$. Let $B_M$ be the set of blank edges in $M$ that have no endpoint
in $A_1\cup A_2$. An alternating path $P$ is {\it relevant} if the
first edge of $P$ is not in $M$, every edge of $P$ is
monochromatic, no edge in $E(P)\setminus M$
is assigned a colour in $C$ and no edge in $E(P)\cap M$
is incident with a vertex in $A_1\cup A_2$. 

\begin{claim}\label{paths}
  For every vertex $v\in A_1$ and every vertex $w\in A_2$ there are at
  least three relevant alternating paths of length $3$ from $v$ to $w$.
\end{claim}

\begin{proof}
Let $v$ be an arbitrary vertex of $A_1$ and let $w$ be an arbitrary
vertex of $A_2$. There are at least $\mu(v)-|C|$ monochromatic edges
incident to $v$ that have a colour in $\overline{C}$. Note that none
of these edges has an endpoint in $A_2$ since otherwise there would be
an alternating path from $v_1$ to $w_1$ including the edge, in which
every other edge was blank, and by switching on this path we
could improve $M$. Also none of these edges is incident to a vertex
which is unmatched by $M$, by a similar argument. There are $|B_M|$
edges from $v$ to vertices in $\W\setminus A_2$ that are matched by 
blank edges of $M$ to vertices in $\V\setminus A_1$. There are $\mu'(v)$
monochromatic edges between $v$ and vertices that are matched by $M$
to vertices in $A_1$. Hence, there are at least
$\mu(v)-|C|-|B_M|-\mu'(v)$ relevant alternating paths of length 2
beginning at $v$.  By symmetry, there are at least
$\mu(w)-|C|-|B_M|-\mu'(w)$ relevant alternating paths of length 2
beginning at $w$. Let $E_v$ be the set of edges of $M$ that are in
relevant alternating paths of length 2 beginning at $v$ and let $E_w$
be the set of edges of $M$ that are in relevant alternating paths of
length 2 beginning at $w$.
Using the previous claims, we see that
\begin{align*}
|E_v\cap E_w|&= |E_v|+|E_w|-|E_v\cup E_w|\\
&\geq \mu(v)-\mu'(v)+\mu(w)-\mu'(w)-2(|C|+|B_M|)-(n+1-|A_1|-|A_2|)\\
&\geq \mu(v)-\mu'(v)+\mu(w)-\mu'(w)-3(n+1-|A_1|-|A_2|)\\
&\geq 3|A_1|+n-\beta(v)-\pi(v)-\mu'(v)+3|A_2|+n-\beta(w)-\pi(w)-\mu'(w)-3n-3\\
&\geq 3+2\beta(v)+2\mu'(v)-\pi(v)+2\beta(w)+2\mu'(w)-\pi(w)-n\\
&\geq 3+2(n/4+\pi(v))-\pi(v)+2(n/4+\pi(w))-\pi(w)-n\\
&\geq 3.
\end{align*}
For each edge in $E_v\cap E_w$, there is a corresponding relevant
alternating path of length 3 from $v$ to $w$, and these paths are
internally disjoint.
\end{proof}

\begin{claim}\label{final}
  No edge of $M$ is in more than one relevant alternating path of
  length $3$ from a vertex in $A_1$ to a vertex in $A_2$.
\end{claim}

\begin{proof}
Assume for contradiction that two such paths $P$ and $Q$ have the same
middle edge. Let $P=p_sxyp_t$ and let $Q=q_sxyq_t$, with $p_s,q_s\in A_1$
and $p_t,q_t\in A_2$. By the definition of $A_1$ and $A_2$, we can
extend $P$ to an alternating path from $v_1$ to $w_1$ in
which all but the three edges of $P$ are blank. Switching on this
alternating path would produce a light matching $M'$ that is larger
than $M$ and has the same number of blank edges.  By our choice of $M$
we know that $M'$ cannot be rainbow, and the reason must be that the
edges $p_sx$ and $yp_t$ are assigned the same colour. By a similar
argument, $q_sx$ and $yq_t$ are assigned the same colour. Since $P$
and $Q$ are distinct paths and each colour defines a matching, $p_sx$
and $yq_t$ are assigned different colours. By the definition of $A_1$
and $A_2$, the path $P'=p_sxyq_t$ can be extended to an alternating
path from $v_1$ to $w_1$ in which all but the three edges of $P'$ are
blank. By switching on this path, we can increase $|M|$ without
reducing the number of blank edges in $M$, a contradiction.
\end{proof}

\begin{proof}[Proof of \tref{t:newfn2}]
By Claims \ref{paths} and \ref{final}, there are at least
$3|A_1||A_2|$ edges in $M$ with no endpoint in $A_1\cup A_2$. So
$3|A_1||A_2|\le n+1-|A_1|-|A_2|$, by Claim \ref{C}. For $i\in \{1,2\}$ we have
$|A_i|\geq n/4 +1$ by Claims \ref{sizeA} and \ref{sizeb}. Hence
\begin{align*}
0&\ge 3|A_1||A_2|+|A_1|+|A_2|-n-1
\ge 3\left(\frac{n}{4}+1\right)^2+2\left(\frac{n}{4}+1\right)-n-1
=\frac{3n^2}{16}+n+4,
\end{align*}
which is false for all positive $n$.
\end{proof}

\begin{corollary}\label{cy:3n/4}
$f(n,2)>{3n}/{4}$ for all positive integers $n$.
\end{corollary}

\begin{proof}
Let $S\subseteq\sym_n$ with $|S|\leq 3n/4$.  By \tref{t:newfn2}, there
is a rainbow perfect matching $M$ with respect to $\mm_S$. Let $p$ be
the permutation such that $M=\{v_iw_{p(i)}\}$. Since $M$ is rainbow, 
$p$ has Hamming distance at least $n-1$ from every permutation in $S$.
\end{proof}



\section{Bounds on $f(n,s)$ for general $s$}\label{s:gens}

It is trivial that $f(n,n)=f(n,n-1)=n!$.  In this section we prove new
upper and lower bounds on $f(n,s)$ for $3\le s\le n-2$.  Note that
Keevash and Ku \cite{KK06} (with a subsequent correction by Aw
\cite{Aw14}) used probabilistic methods to provide a lower bound
for the covering radius of sets of
permutations $S\subseteq\sym_n$ for which no edge occurs in more than
a given number $k$ of the matchings in $\mm_S$.

Let $B(n,k)$ be the number of permutations in the ball of radius $k$
around the identity in $\sym_n$. Then
\[
B(n,k)=\sum_{i=0}^k\binom{n}{i}d_i\approx \frac{n!}{e}\sum_{i=0}^k\frac1{(n-i)!},
\]
where $d_i$ is the number of derangements in $\sym_i$.
In \cite{CW05} it was noted that
$f(n,s)\ge n!/B(n,n-s)$. We now show a similar upper bound, involving
the Harmonic number 
$H_b={1}+\frac{1}{2}+\dots+\frac{1}{b}\approx \log b+\gamma$.
We also need the notion of a {\it fractional cover}, which is an
assignment of real weights to the vertices of a hypergraph that
results in every hyperedge having total weight of at least 1. The {\it
  fractional covering number} $\tau^*$ is the minimum total weight of
a fractional cover. A {\it cover} is a fractional cover in which every
weight is 0 or 1.  The {\it covering number} $\tau$ is the minimum
total weight of a cover.

\begin{thrm}
Let $b=B(n,n-s)$ for $1\le s\le n$. Then $n!/b \le f(n,s) \le H_bn!/b$.
\end{thrm}

\begin{proof}
We construct a hypergraph $H$ whose vertices are the permutations in $\sym_n$.
For each $\sigma\in\sym_n$ we add an edge to $H$ which consists of all
of the permutations in the ball of radius $n-s$ around $\sigma$.
The resulting hypergraph is $b$-regular and $b$-uniform, 
so its fractional covering number $\tau^*$ is $n!/b$, with an optimal
covering being obtained by assigning a weight of $1/b$ to every vertex.
By construction, $f(n,s)$ is the covering number $\tau$ of $H$.
By a result of Lov\'asz~\cite{Lov75}, we know that $\tau^*\le\tau\le\tau^*H_b$,
from which our theorem follows.
\end{proof}

\begin{corollary}
$\frac{2s}{s+1}s!< f(n,s)< 3s!(n-s)\log n$ for $1\le s\le n-2$.
\end{corollary}

\begin{proof}For the upper bound we use that 
\begin{align}\label{e:blowbnd}
\frac{b}{n!}
&=\sum_{i=0}^{n-s}\frac{d_i}{i!(n-i)!}
>\frac{d_{n-s}}{(n-s)!s!}
=\frac1{s!}\sum_{j=0}^{n-s}\frac{(-1)^j}{j!}
\ge\frac1{3s!}
\end{align}
given that $n-s\ge2$. Also, $H_b<\log(3b)$ for integers $b\ge1$
and $\log(3b)/b$ is monotone decreasing for real $b\ge1$. 
So by \eref{e:blowbnd},
\begin{align*}
{n!H_b}/{b}&<{n!\log(3b)}/{b} < 3s!\log(n!/s!) < 3s!(n-s)\log n.
\end{align*}
For the lower bound we have
\begin{align*} 
\frac{b}{n!}
&=\frac{1}{n!}+\sum_{i=2}^{n-s}\frac1{(n-i)!}\sum_{j=0}^i\frac{(-1)^j}{j!}\\
&\le\frac{1}{n!}+\sum_{i=2}^{n-s}\frac1{2(n-i)!}\\
&<\frac12\sum_{i=s}^{n}\frac1{i!}\\
&\le\frac{1}{2s!}\sum_{i=0}^{n-s}\frac1{(s+1)^i}\\
&<\frac{s+1}{2s\,s!}.
\end{align*}
From which it follows that $n!/b>s!\,2s/(s+1)$.
\end{proof}

\begin{corollary}
$s!< f(n,s)< s!\,O(n \log n)$ for $n\iwb$ with $1\le s\le n-2$.
If $n-s$ is bounded then $s!< f(n,s)< s!\,O(\log n)$.
\end{corollary}

It was shown in \cite{CW05} that $f(n,s)< s!\,O(n \log n)$ 
provided $n\ge2s+2$, but we have shown that this extra assumption is
not required. For a probabilistic argument that shows that $f(n,s)\ge s!$,
see \cite[Prop.~3.1]{KK06}.

We next prove a result that will provide 
a lower bound on $f(n,s)$ for all $s\geq 3$. 


\begin{thrm}\label{mainthm2}
Let $G$ be the complete bipartite graph $K_{n,n}$, 
let $\alpha=\lfloor\sqrt{2s-2}/2+1\rfloor/2$
for an integer $s\ge3$
and let $\mm_S$ be a set of at most $\alpha n$ perfect matchings of
$G$. There is a perfect matching of $G$ which is $(s-1)$-rainbow and
$(2\alpha-1)$-light with respect to $\mm_S$.
\end{thrm}

We adopt a very similar approach and notation to that used in the
proof of \tref{t:newfn2}. 



\begin{claim}
There exists a $(2\alpha-1)$-light perfect matching of $G$.
\end{claim}

\begin{proof}
Let $G'$ be the graph formed from $G$ by deleting all edges with at
least $2\alpha$ colours. Note that since $G$ is $n$-regular and
$|\mm_S|\leq \alpha n$, we have minimum degree $\delta(G')\geq n/2$.  Let
$A$ be an arbitrary subset of $\V$ and let $N$ denote the set of
neighbours of $A$ in $G'$. If $A=\emptyset$, then $|A|=|N|=0$. If
$1\leq A\leq n/2$, then $|N|\geq\delta(G')\geq n/2\geq |A|$.  So
suppose that $|A|>n/2$ and let $w$ be an arbitrary vertex of $\W$. Since
$|\V\setminus A|<n/2\le\delta(G')$, we see that $w$ must have a neighbour
in $A$. Hence, $|N|=|\W|=n\geq|A|$.  By Hall's marriage theorem, $G'$ has
a perfect matching, which is a $(2\alpha-1)$-light perfect matching
of $G$.
\end{proof}

Let $M$ be a $(2\alpha-1)$-light perfect matching of $G$ which
minimises the number of edges with a colour that appears on at least
$s$ edges, and subject to this condition, maximises the number of
blank edges. Aiming for a contradiction, suppose that some colour $c$
appears on at least $s$ edges of $M$. Without loss of generality,
$v_1w_1$ is an edge of $M$ that has colour $c$.

Let $P$ be a maximum length alternating path containing the edge
$v_1w_1$ such that the first edge and last edge of $P$ are in $M$ and
every edge in $E(P)\setminus M$ is blank. We can be sure that $P$ exists,
since the path consisting of the single edge $v_1w_1$ satisfies the
requirements. Now, $P$ has an odd number of edges, so we may assume that
it starts at a vertex $v_a\in\V$ and ends at a vertex $w_b\in\W$. 
Let $P_1$ be the subpath of $P$ from $v_a$ to $v_1$ and let $P_2$ be
the subpath of $P$ from $w_1$ to $w_b$. Note that $P_1$ has even
length. If $P_1$ has positive length then its first edge is in $M$ and its last 
edge is blank and cannot be $v_1w_1$. A similar argument applies to $P_2$, 
so $P_1$ and $P_2$ are disjoint.


Let $C_1$ be the set of colours assigned to exactly $s-2$ edges in
$M\setminus \{v_1w_1\}$, and let $C_2$ be the set of colours assigned
to at least $s-1$ edges in $M\setminus \{v_1w_1\}$. Let $G^*$ be the graph
obtained from $G$ by removing the edges in $M$.
Let $A(v_a)$ be
the set of edges incident to $v_a$ in $G^*$ which are assigned between
1 and $2\alpha-1$ colours and such that none of these colours is in
$C_1\cup C_2$. Let $A(w_b)$ be the set of edges incident to $w_b$ in
$G^*$ which are assigned between 1 and $2\alpha-1$ colours and such
that none of these colours is in $C_2$. For $u\in \{v_a,w_b\}$, let
$B(u)$ be the set of blank edges incident to $u$ in $G^*$.  Let
$\psi=|B(v_a)\cup A(v_a)\cup B(w_b)\cup A(w_b)|$.

Let $e_1$ be the edge in $P$ incident to $v_a$ and let $e_2$ be the
edge in $P$ incident to $w_b$. Let $\kappa_1=0$ if $e_1$ is assigned
at least one colour in $C_1\cup C_2$, and otherwise let
$\kappa_1=1$. Let $\kappa_2=0$ if $e_2$ is assigned at least one
colour in $C_2$, and otherwise let $\kappa_2=1$. Let
$\kappa=\kappa_1+\kappa_2$.


\begin{claim}\label{UpperBound}
$B(v_a)$, $A(v_a)$, $B(w_b)$ and $A(w_b)$ are disjoint sets and
$\psi\leq n-1-\kappa$.
\end{claim} 

\begin{proof}
It is immediate from their definitions that $B(v_a)$, $A(v_a)$,
$B(w_b)$ and $A(w_b)$ are disjoint sets unless
$v_aw_b\in\big(B(v_a)\cup A(v_a)\big)\cap\big(B(w_b)\cup A(w_b)\big)$.


Suppose that there is an edge $v_au\in B(v_a)\cup A(v_a)$ such that
$u\in V(P)$ and the subpath $P'$ of $P$ from $v_a$ to $u$
contains the edge $v_1w_1$. Note that the only edge incident to $v_a$
in $P$ is in $M$, so $v_au$ is not in $P'$ and adding the edge $v_au$
to $P'$ creates an alternating cycle. Every edge in this alternating
cycle is assigned at most $2\alpha-1$ colours, and switching on this
cycle reduces the number of edges in $M$ which have a colour that
appears on at least $s$ edges of the matching, contradicting our
choice of $M$.

We note two consequences. Firstly,
$v_aw_b\notin \big(B(v_a)\cup A(v_a)\big)\cap\big(B(w_b)\cup A(w_b)\big)$, so
$B(v_a)$, $A(v_a)$, $B(w_b)$ and $A(w_b)$ are disjoint sets. Secondly, no edge in
$B(v_a)\cup A(v_a)$ is incident to a vertex in $P_2$. A similar argument shows
that no edge in $B(w_b)\cup A(w_b)$ is incident to a vertex in
$P_1$. However, every edge that is not in $M$ is incident with two edges in $M$,
because $M$ is a perfect matching.
It follows that every edge in $B(v_a)\cup A(v_a)\cup B(w_b)\cup
A(w_b)$ is incident to some edge of $M\setminus \{v_1w_1\}$ that is
neither the first nor the last edge of $P$. Note that $\kappa\in
\{0,1,2\}$, and if $v_1w_1$ is the only edge of $P$, then
$\kappa=0$. Furthermore, if $\kappa=2$, the first and last edges of
$P$ are distinct edges in $M\setminus \{v_1w_1\}$.



So if we suppose for contradiction that $\psi>n-1-\kappa$
then there is some edge $v_xw_y\in M$ that is neither the first edge nor the last edge of $P$
and for which $v_xw_b\in B(w_b)\cup A(w_b)$ 
and $v_aw_y\in B(v_a)\cup A(v_a)$. Therefore, 
$v_x\notin V(P_1)$ and $w_y\notin V(P_2)$, so $v_xw_y$ is not an edge of $P$.
Let $\Gamma$ be the cycle formed by combining $P$ and the path
$w_bv_xw_yv_a$.
Then
$\Gamma$ is an alternating cycle which contains $v_1w_1$ and no edges
that have more than $2\alpha-1$ colours. There are at most two
non-blank edges in $E(\Gamma)\setminus M$, and no colour that appears
on $E(\Gamma)\setminus M$ appears more than $s-1$ times in
$(E(\Gamma)\cup M)\setminus \{v_1w_1\}$. Hence, the matching obtained
by switching on $\Gamma$ contradicts our choice of $M$.
\end{proof}


\begin{claim}\label{UsedColours}
$(s-2)|C_1|+(s-1)|C_2|\leq (2\alpha-1)(n-1-|B(v_a)|-|B(w_b)|-\kappa_1)$.
\end{claim}

\begin{proof} 
Let $w'$ be a vertex in $\W$ such that $v_aw'\in B(v_a)$, and let $v'$
be the vertex matched to $w'$ by $M$. If $w'\notin V(P)$, then $v'\notin V(P)$ 
so combining $v'w'v_a$ with $P$ creates a longer alternating
path containing $v_1w_1$ such that the first and last edges are in $M$
and the edges not in $M$ are blank, contradicting our choice of
$P$. Hence, $w'\in V(P)$. Let $P'$ be the subpath of $P$ from $v_a$ to
$w'$ and let $\Gamma$ be the cycle obtained from $P'$ by adding
$w'v_a$. Note that switching on $\Gamma$ does not add any coloured
edge to the matching. Hence, by our choice of $M$, every edge of
$\Gamma$ is blank. In particular, $v'w'$ is blank, and since $v_1w_1$
is not blank, $w'\in V(P_1)$. So, every edge in $B(v_a)$ connects
$v_a$ to a blank edge in $E(P_1)\cap M$. By symmetry, every edge in
$B(w_b)$ connects $w_b$ to a blank edge in $E(P_2)\cap M$. Hence, there
are at most $n-1-|B(v_a)|-|B(w_b)|-\kappa_1$ edges in
$M\setminus \{v_1w_1\}$ with a colour in $C_1\cup C_2$.
Every colour in $C_1$ appears on $s-2$ of
these edges and every colour in $C_2$ appears on at least $s-1$ of
these edges. Since every edge in $M$ is assigned at most $2\alpha-1$
colours, the claim follows.
\end{proof}

Let $D(v_a)$ be the set of edges incident to $v_a$ that are assigned
at least $2\alpha$ colours not in $C_1\cup C_2$, and let $D(w_b)$ be
the set of edges incident to $w_b$ that are assigned at least
$2\alpha$ colours not in $C_2$.

\begin{claim}\label{LowerBoundA}
$|B(v_a)\cup A(v_a)|\geq n-|D(v_a)|-|C_1|-|C_2|-\kappa_1$.
\end{claim}

\begin{proof}
There are $n$ edges incident to $v_a$ in $G$. An edge incident to
$v_a$ is in $B(v_a)\cup A(v_a)$ unless it has a colour in $C_1\cup C_2$, 
it has at least $2\alpha$ colours, or it is in $M$. There are
exactly $|C_1|+|C_2|$ edges incident to $v_a$ with a colour in
$C_1\cup C_2$, since each colour corresponds to a perfect matching and
$C_1$ and $C_2$ are disjoint. The number of edges incident to $v_a$
with at least $2\alpha$ colours that are not in $C_1\cup C_2$ is 
$|D(v_a)|$. The number of edges incident to $v_a$ that are in $M$
and have no colour in $C_1\cup C_2$ is $\kappa_1$.
\end{proof}

\begin{claim}\label{LowerBoundB}
$|B(w_b)\cup A(w_b)|\geq n-|D(w_b)|-|C_2|-\kappa_2$.
\end{claim}

\begin{proof}
There are $n$ edges incident to $w_b$ in $G$. An edge incident to
$w_b$ is in $B(w_b)\cup A(w_b)$ unless it has a colour in $C_2$, it
has at least $2\alpha$ colours, or it is in $M$. There at exactly
$|C_2|$ edges incident to $w_b$ with a colour in $C_2$, since each
colour corresponds to a perfect matching. The number of edges incident
to $w_b$ with at least $2\alpha$ colours that are not in $C_2$ is
$|D(w_b)|$. The number of edges incident to $w_b$ that are in $M$ and
have no colour in $C_2$ is $\kappa_2$.
\end{proof}

\begin{claim}\label{MulticolourBoundA}
$|D(v_a)|\leq \frac{1}{2\alpha}(\alpha n-|C_1|-|C_2|-|A(v_a)|)$.
\end{claim}

\begin{proof}
The total number of colours not in $C_1\cup C_2$ and not used by the edges in $A(v_a)$ is at most $\alpha n-|C_1|-|C_2|-|A(v_a)|$. Each colour appears on at most one edge incident to $v_a$. Therefore the result follows from the definition of $D(v_a)$.
\end{proof}

\begin{claim}\label{MulticolourBoundB}
$|D(w_b)|\leq \frac{1}{2\alpha}(\alpha n-|C_2|-|A(w_b)|)$.
\end{claim}

\begin{proof}
The total number of colours not in $C_2$ and not used by the edges in $A(w_b)$ is at most $\alpha n-|C_2|-|A(w_b)|$. Each colour appears on exactly one edge incident to $w_b$. Therefore the result follows from the definition of $D(w_b)$.
\end{proof}

\begin{proof}[Proof of \tref{mainthm2}]
Since $\alpha=\lfloor \sqrt{2s-2}/2+1\rfloor/2$ and $s\geq 3$, we have $1\ge 1/(2\alpha)\ge(2\alpha-1)^2/(s\alpha-\alpha)$ with at least one of the inequalities being strict. By \cref{LowerBoundA,LowerBoundB,UsedColours,MulticolourBoundA,MulticolourBoundB},
\begin{align*}
\psi&\geq 2n-|D(v_a)|-|D(w_b)|-|C_1|-2|C_2|-\kappa\\
&\geq 2n-\frac{1}{ 2\alpha}(2\alpha n-|C_1|-2|C_2|-|A(v_a)|-|A(w_b)|)-|C_1|-2|C_2|-\kappa\\
&= n+\frac{1}{2\alpha}(|A(v_a)|+|A(w_b)|)-\frac{(2\alpha-1)}{2\alpha}(|C_1|+2|C_2|)-\kappa\\
&\geq n+\frac{1}{2\alpha}(|A(v_a)|+|A(w_b)|)-\frac{(2\alpha-1)}{(s-1)\alpha}\big((s-2)|C_1|+(s-1)|C_2|\big)-\kappa\\
&\geq n+\frac{1}{ 2\alpha}(|A(v_a)|+|A(w_b)|)-\frac{(2\alpha-1)^2}{(s-1) \alpha}(n-1-|B(v_a)|-|B(w_b)|-\kappa_1)-\kappa\\
&\geq \left(1-\frac{( 2\alpha-1)^2}{(s-1) \alpha}\right)n+\frac{( 2\alpha-1)^2}{(s-1) \alpha}(\psi+1+\kappa_1)-\kappa.
\end{align*}
Hence
\begin{align*}
\psi&\geq n+\frac{( 2\alpha-1)^2}{(s-1) \alpha-( 2\alpha-1)^2}(1+\kappa_1)-\frac{(s-1) \alpha\kappa}{(s-1) \alpha-( 2\alpha-1)^2}\geq n-\kappa,
\end{align*}
given that $1+\kappa_1\ge\kappa$. 
Claim~\ref{UpperBound} now provides a contradiction
to our assumption that some colour appears on at least $s$ edges of $M$,
completing the proof of the theorem.
\end{proof}

\begin{corollary}\label{cy:linlowbnd}
If $s\geq 3$, then 
\[f(n,s)>\left\lfloor \frac{2+\sqrt{2s-2}}{2}\right\rfloor \frac{n}{2}.\]
\end{corollary}


In light of \cjref{cj:KS}, it is worth remarking that
\cyref{cy:linlowbnd} shows that $f(n,3)>n$.  It is immediate from the
definition that $f(n,3)\ge f(n,2)$, but we now know that this
inequality is strict for even $n$. We also know that the rows of any
row-latin square form a set of permutations with covering radius at
least $n-2$ (a result recently obtained independently by Aharoni
\emph{et al.} \cite[Thm~1.16]{ABKZ18}).

\subsection*{Acknowledgements}

The authors are grateful to Ron Aharoni for bringing reference \cite{Lov75}
to their attention.

 
  \let\oldthebibliography=\thebibliography
  \let\endoldthebibliography=\endthebibliography
  \renewenvironment{thebibliography}[1]{%
    \begin{oldthebibliography}{#1}%
      \setlength{\parskip}{0.2ex}%
      \setlength{\itemsep}{0.2ex}%
  }%
  {%
    \end{oldthebibliography}%
  }

\end{document}